\documentclass[11pt]{article}

\usepackage[margin=1in]{geometry}
\usepackage{amsmath, amssymb, amsthm}
\usepackage{graphicx}
\usepackage{enumitem}
\usepackage[expansion=false]{microtype}
\usepackage[T1]{fontenc}

\newtheorem{theorem}{Theorem}
\newtheorem{definition}{Definition}
\newtheorem{remark}{Remark}

\renewcommand{\thefootnote}{\fnsymbol{footnote}}

\title{Counting Hamiltonian Paths in Recursive 3-Regular Planar Graphs%
  \protect\footnotemark[2]}

\author{%
  Ira Pohl\textsuperscript{1} \and
  Larry Stockmeyer\protect\footnotemark[1]\textsuperscript{2}%
}

\date{%
  \textsuperscript{1}Department of Computer Science,
  University of California, Santa Cruz\\
  \textsuperscript{2}Formerly IBM Almaden Research Center\\[1ex]
  Revised May 20, 2026%
}

\begin{document}

\maketitle

\footnotetext[1]{Deceased, 2004.  Joint work begun before Larry's
death and previously published in part as~[14]; this paper completes
the BT analysis to an exact closed form.}
\footnotetext[2]{Claude 4.7, an AI assistant developed by Anthropic,
assisted in revising this paper.}

\renewcommand{\thefootnote}{\arabic{footnote}}
\setcounter{footnote}{0}

\begin{abstract}
We introduce two infinite families of $3$-regular planar graphs, the
\emph{recursive Tutte-style graphs} $RT_k$ and the \emph{binary-tree graphs}
$BT_k$, designed as benchmark instances for Hamiltonian-path heuristics.
Both families are conceptual adversaries to the Pohl--Warnsdorf rule and
related local methods: the regularity removes any degree signal, and the
recursive structure creates bottlenecks where greedy paths trap.  A
closed-form expression for $\mathrm{Ham}(RT_k)$ and bounds on
$\mathrm{Ham}(BT_k)$ were given in our earlier paper~[14].  Here we
sharpen the $BT$ analysis to an exact closed form,
\[
  \mathrm{Ham}(BT_k) \;=\;
    \tfrac{3}{16}\bigl(17\,V_{2k} + 71\,W_{2k}\bigr)
    \;-\; \tfrac{9}{2}\cdot 4^{k},
\]
where $V_j,W_j$ are integer sequences associated with
$\alpha=(1+\sqrt{17})/2$; the dominant growth is $\alpha^{2k}$, of
polynomial degree $\log_2 \alpha^{2}\approx 2.714$.  An appendix sketches
a conjectural transfer-matrix generalization (further work).
\end{abstract}

\section{Introduction}

The Pohl--Warnsdorf rule is a greedy heuristic for finding long simple
paths in a graph, based on a 19th-century rule of H.~C.~von Warnsdorf
for the knight's tour: from the current vertex, move to a neighbor of
least remaining degree.  In~[11] we observed that the rule generalizes
to arbitrary graphs and proposed a recursive tie-breaking refinement
(the Pohl--Warnsdorf rule proper), demonstrating its effectiveness on
knight's-tour problems and on Tutte's celebrated $46$-vertex graph~[17],
a $3$-regular planar graph that disproved Tait's 1884 conjecture for
the four-color problem.

Tutte's graph (Figure~\ref{fig:tutte}) was chosen because it is a
conceptual adversary to local heuristics: its $3$-regularity eliminates
the degree signal that Warnsdorf's rule depends on, so its only initial
guidance comes from the rule's recursive tie-breaking.  Modern hardware
permits experiments at scales unreachable in 1967, prompting us to
revisit Tutte's graph and look for systematic ways to generate similarly
adversarial benchmarks.  This program was begun in joint work with
Larry Stockmeyer; an initial account, including the $RT$ and $BT$
constructions, the closed form for $\mathrm{Ham}(RT_k)$, and bounds on
$\mathrm{Ham}(BT_k)$, appeared in~[14].  The present paper extends that
work.

\begin{figure}[h]
  \centering
  \includegraphics[width=0.45\textwidth]{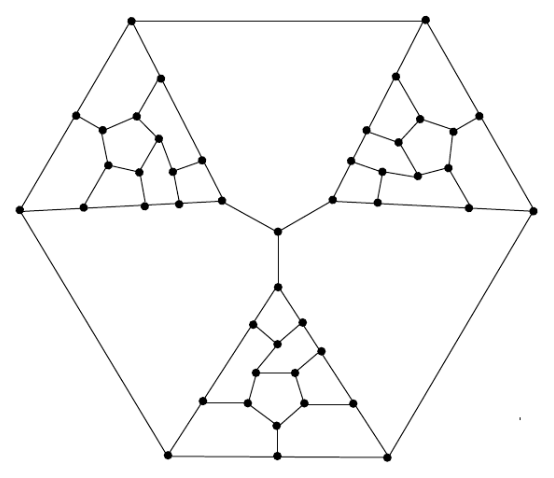}
  \caption{Tutte's graph.}
  \label{fig:tutte}
\end{figure}

\paragraph{Two new families.}
Tutte's graph admits a $3$-fold decomposition: it contains three
pairwise-isomorphic $15$-vertex subgraphs that, when each is contracted
to a super-node, leave the complete graph $K_4$.  This structural
observation suggests recursive constructions in which $K_4$ is
repeatedly refined by replacing each non-center vertex with a small
isomorphic component.  We follow this idea in two distinct directions,
obtaining the families $RT_k$ (triangle-based) and $BT_k$
(binary-tree-based).  Both are infinite series of $3$-regular planar
graphs with simple recursive definitions.

\paragraph{Results.} The main results of this paper are:
\begin{enumerate}[label=(\roman*)]
  \item A closed-form formula for $\mathrm{Ham}(RT_k)$ (Theorem~\ref{thm:rt}),
        recapitulated from~[14], giving polynomial growth of degree
        $\log_3 16\approx 2.524$.
  \item A closed-form formula for $\mathrm{Ham}(BT_k)$
        (Theorem~\ref{thm:bt}) in terms of Lucas-like sequences associated
        with the quadratic $x^{2}-x-4=0$.  This sharpens the bounds
        of~[14] to an exact identity, by recognizing that the cubic
        characteristic polynomial of that analysis factors.
\end{enumerate}
Appendix~B sketches a conjectural transfer-matrix generalization of
both results; this material is the responsibility of the first author
alone and is included as a record of further work in progress.

\paragraph{Significance.}
Beyond their role as heuristic benchmarks, the graphs have two
structural features of independent interest.  First, they realize
fractal-like self-similarity in the cubic planar setting, with $RT'_{k}$
resembling a Sierpi\'nski gasket but with all internal vertices of
degree~$3$.  Second, they exhibit \emph{polynomial} Hamiltonian-path
density---a regime that contrasts with the $2^{3n/8}$ upper bound of
Eppstein~[4] for Hamiltonian cycles in cubic graphs and the
$\Theta(\alpha^{n})$ growth obtained by extremal constructions.  As a
counting class, they sit between sparse graphs (linear Hamiltonian-path
counts) and Hamiltonian-rich constructions (exponential counts).

\paragraph{Organization.}
Section~\ref{sec:algorithm} recalls the Pohl--Warnsdorf algorithm.
Section~\ref{sec:families} defines the two families.
Section~\ref{sec:results} states the closed-form theorems.
Section~\ref{sec:related} surveys related work; Section~\ref{sec:open}
lists open problems.  The proof of Theorem~\ref{thm:rt} is in~[14]; the
proof of Theorem~\ref{thm:bt}, which is new, appears in
Section~\ref{sec:proof}.  Appendix~A reproduces the original Algol~60
code from~[11]; Appendix~B sketches a conjectural transfer-matrix
generalization.

\section{The Pohl--Warnsdorf Algorithm}
\label{sec:algorithm}

For completeness we recall the algorithm; details are in~[11,13].
From a chosen start vertex, repeatedly extend the current path by
moving to an unvisited neighbor of \emph{minimum} current degree
(counting only unvisited neighbors).  Ties are broken recursively:
among candidates of equal minimum degree, prefer one whose own
minimum-degree unvisited neighbor (other than the current vertex) has
smallest degree.  Further ties are broken arbitrarily.

The algorithm is fast (polynomial per start vertex) and surprisingly
effective on sparse graphs with nontrivial degree structure---e.g.,
the knight's-graph on $n\times n$ boards.  On $3$-regular graphs,
however, the initial degree signal is constant, so all guidance comes
from tie-breaking.  This makes $3$-regular planar graphs natural
stress tests.

\section{Two Recursive Families of $3$-Regular Planar Graphs}
\label{sec:families}

We now recall the two families introduced in~[14].  Both have a simple
substitution-based definition; both are $3$-regular and planar; both
are constructed so that recursive depth forces a path to commit early
to traversals of large sub-components, with only a few exit ports
available.

\subsection{The Recursive Tutte-Style Graphs $RT_k$}

We define $RT_k$ via auxiliary graphs $RT'_k$ in which three
distinguished vertices (\emph{ports}) have degree~$2$; the remaining
vertices have degree~$3$.

\begin{definition}
Let $RT'_1 := K_3$ with all three vertices designated as ports.  For
$k\ge 2$, $RT'_k$ is formed from three disjoint copies of $RT'_{k-1}$,
labeled $D,E,F$, by adding three edges: one between a port of $D$ and
a port of $E$, one between $E$ and $F$, and one between $F$ and $D$.
The three remaining (unconnected) ports become the ports of $RT'_k$.
The graph $RT_k$ is then $RT'_k$ together with one additional
\emph{center} vertex $c$ and three edges from $c$ to the three ports
of $RT'_k$.
\end{definition}

It is immediate that $RT_k$ is $3$-regular, and that planarity is
preserved by drawing the three $RT'_{k-1}$ sub-components in a
rotationally symmetric configuration with the center vertex in the
middle.  The vertex count satisfies $|V(RT_k)| = 3^{k}+1$, giving the
sequence $4,10,28,82,244,730,\ldots$.  In particular $RT_1 = K_4$.

\begin{figure}[h]
  \centering
  \includegraphics[width=0.85\textwidth]{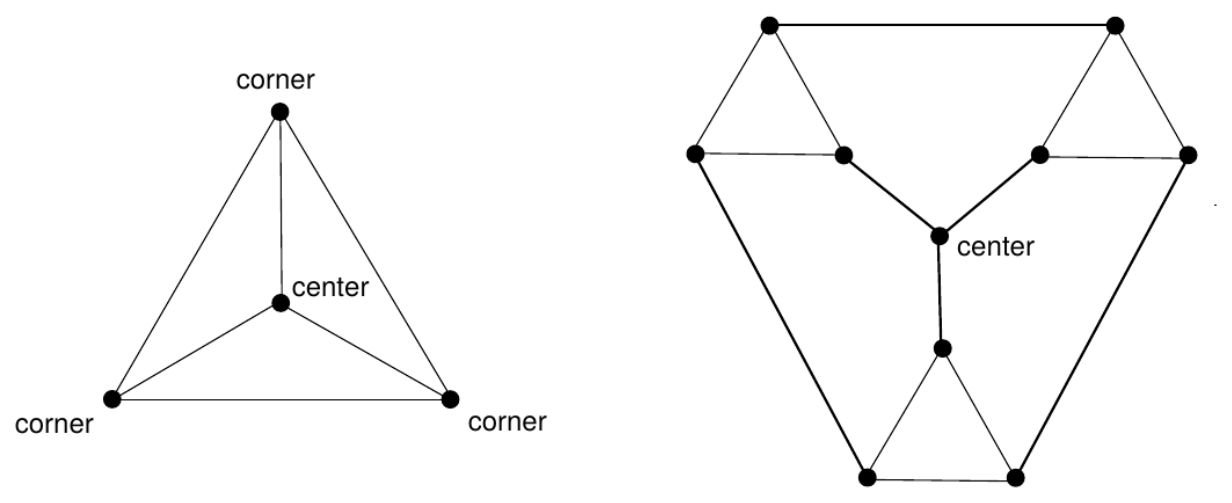}
  \caption{$K_4$ labeled with corners and center (left).  The $RT_2$
  graph (right) is obtained by replacing each corner of $K_4$ by a
  triangle.}
  \label{fig:rt2}
\end{figure}

\begin{figure}[h]
  \centering
  \includegraphics[width=0.95\textwidth]{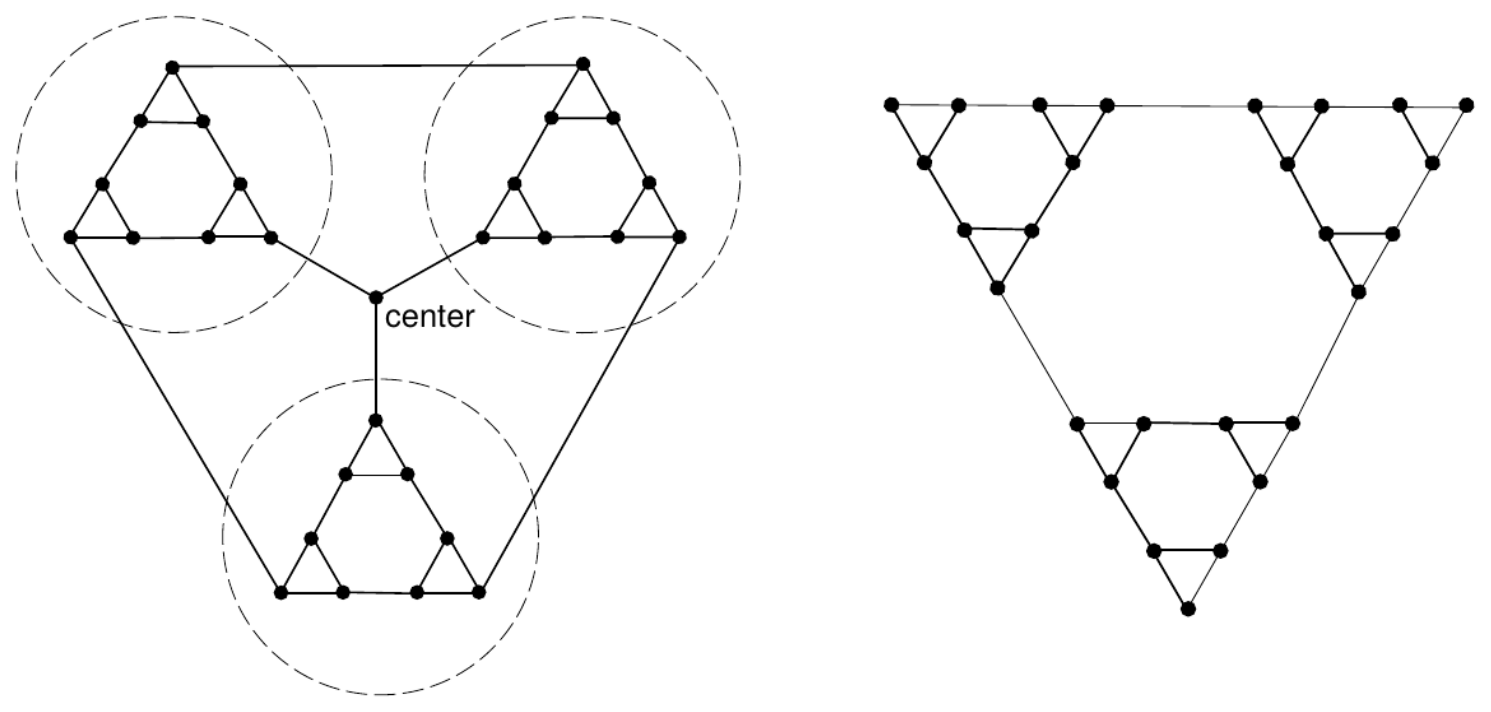}
  \caption{The $RT_3$ graph (left).  The component within each dotted
  circle is a copy of $RT'_2$, where $RT'_2$ is $RT_2$ with its center
  removed.  The $RT'_3$ graph (right) is obtained by removing the
  center from $RT_3$.  Each component $RT'_2$ in $RT'_3$ has been
  flipped to maintain planarity at the next step, $RT_4$.}
  \label{fig:rt3}
\end{figure}

\paragraph{Self-similarity.}
Looking at $RT'_3$ in Figure~\ref{fig:rt3} (right), the graph has a
self-similar, or fractal, nature.  The Sierpi\'nski triangle, introduced
by Sierpi\'nski in 1915~[16], is the classical fractal obtained from an
equilateral triangle by repeatedly removing the central inverted triangle
of each remaining sub-triangle; the limit set is the canonical example
of a self-similar planar fractal.  The $RT'$ series appears similar to
the Sierpi\'nski triangle, though with all internal nodes of degree
exactly~$3$ (the Sierpi\'nski gasket viewed as a graph has all internal
vertices of degree~$4$).  There is an alternate, ``fractal-like'' way to
generate the $RT$ series.  Starting from $RT_1 = K_4$, the graph
$RT_{k+1}$ is obtained from $RT_k$ by applying the local transformation
of Figure~\ref{fig:transform} to every vertex of $RT_k$ except the center
vertex.

\begin{figure}[h]
  \centering
  \includegraphics[width=0.55\textwidth]{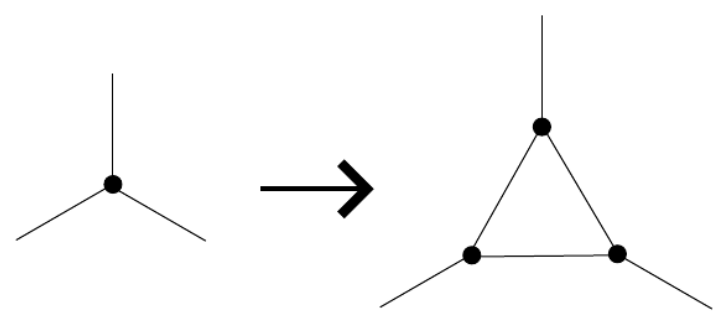}
  \caption{A fractal-like transformation for growing $3$-regular graphs:
  each degree-$3$ vertex is replaced by a triangle whose three vertices
  inherit the three former incident edges.}
  \label{fig:transform}
\end{figure}

\paragraph{The first three members of $RT$.}
$RT_1 = K_4$ ($4$~vertices); $RT_2$ has $10$~vertices (each ``corner''
of $K_4$ becomes a $K_3$ triangle); $RT_3$ has $28$~vertices.

\subsection{The Binary-Tree Graphs $BT_k$}

The second family arises by replacing the recursive triangle clusters
of $RT_k$ with leaf-connected complete binary trees.

\begin{definition}
Let $BT'_1 := K_3$, with one vertex designated the \emph{root port}
and the other two the \emph{leaf ports}.  For $k\ge 2$, $BT'_k$ is
formed from two disjoint copies of $BT'_{k-1}$ (labeled $D,E$) and one
new vertex $q$, by adding three edges: $q$ to the root ports of $D$
and $E$, and one edge between a leaf port of $D$ and a leaf port
of~$E$.  The remaining (unconnected) leaf ports of $D$ and $E$ become
the leaf ports of $BT'_k$, and $q$ becomes its root port.

The graph $BT_k$ is formed from three disjoint copies of $BT'_{k-1}$,
one new center vertex $c$, and six new edges: $c$ to the root port of
each copy, and the three leaf ports of each copy paired with leaf
ports of the other copies so that each leaf port has exactly one new
neighbor.
\end{definition}

Equivalently: $BT'_k$ is a complete binary tree of depth $k$ in which
all $2^k$ leaves are joined left-to-right by a path; the root is the
root port, and the two endpoints of the leaf-path are the leaf ports.
The vertex count of $BT_k$ is $3\cdot 2^{k}-2$, giving
$4,10,22,46,94,190,\ldots$.  In particular $BT_1 = K_4$, and $BT_3$
has $22$~vertices (Figure~\ref{fig:bt}).

\begin{figure}[h]
  \centering
  \includegraphics[width=0.75\textwidth]{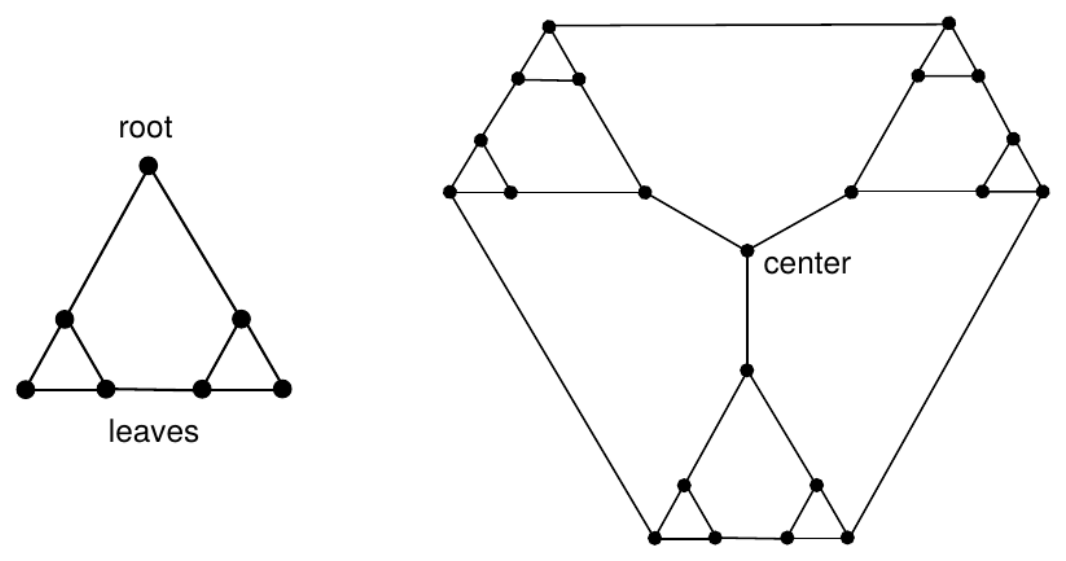}
  \caption{The $BT'_2$ graph (left).  This component of $7$ nodes
  becomes part of the $BT_3$ graph of $22$ nodes (right).}
  \label{fig:bt}
\end{figure}

Our enumerative computer program counted $204$ paths for $BT_2$,
$1524$ paths for $BT_3$, and $10{,}740$ paths for $BT_4$~[10], results
later confirmed by the analysis of~[14] (and now extended to exact
closed form below).

\section{Counting Hamiltonian Paths: Statement of Results}
\label{sec:results}

Throughout this paper, a Hamiltonian path is \emph{oriented} (i.e., a
Hamiltonian path and its reversal are counted separately).  Thus
$\mathrm{Ham}(K_n) = n!$.  The unoriented count is
$\mathrm{Ham}(G)/2$.

We say a collection of vertex-disjoint paths \emph{$H$-covers} a graph
$G$ if their vertex sets partition $V(G)$.  A single path $H$-covers
$G$ iff it is Hamiltonian.  The key technical idea is that to count
Hamiltonian paths in a recursive graph it does not suffice to count
Hamiltonian paths in the sub-components; rather, one must count
constrained collections of paths through the sub-components, classified
by the role of each port.  We refer to such classified counts as
\emph{port-configuration counts}.

\subsection{Closed Form for $\mathrm{Ham}(RT_k)$}

\begin{theorem}[{[14]}]
\label{thm:rt}
For all $k\ge 1$,
\[
  \mathrm{Ham}(RT_k) \;=\;
    \tfrac{8}{13}\cdot 16^{k} \;+\; 2\cdot 4^{k} \;+\;
    \tfrac{18}{13}\cdot 3^{k} \;+\; 2.
\]
Writing $n = 3^{k}+1$ for the vertex count of $RT_k$ and
$\lambda = \log_3 16$,
\[
  \mathrm{Ham}(RT_k) \;=\;
    \tfrac{8}{13}\,(n-1)^{\lambda} \;+\; 2\,(n-1)^{\log_3 4}
    \;+\; \tfrac{18}{13}\,n \;+\; \tfrac{8}{13}.
\]
In particular, $\mathrm{Ham}(RT_k)$ grows polynomially in $n$ with
exponent $\lambda = \log_3 16 \approx 2.5237$.
\end{theorem}

The proof, via the four port-configuration counts $[P{-}P]_k$,
$[P{-}A]_k$, $[P{-}P,P{-}A]_k$, $[P{-}A,P{-}A]_k$ and their
recurrences, appears in [14, Sec.~8].

\subsection{Closed Form for $\mathrm{Ham}(BT_k)$}

The $BT$ analysis in~[14] produces a system of eight port-configuration
recurrences from which Theorem~2 of that paper extracts upper and
lower bounds on $\mathrm{Ham}(BT_k)$ that match to within a constant
factor.  The dominant root of the analysis is the unique real
$\alpha\in[2.5,2.6]$ satisfying the cubic
\[
  \alpha^{3}-2\alpha^{2}-3\alpha+4 = 0.
\]
The new observation that unlocks an exact closed form is that this
cubic factors:
\[
  x^{3}-2x^{2}-3x+4 \;=\; (x-1)(x^{2}-x-4),
\]
so $\alpha$ is in fact algebraic of degree~$2$ over $\mathbb{Q}$, with
$\alpha,\beta = (1\pm\sqrt{17})/2$.  This factorization, combined with
the fact that the system $\{[L{-}A],[R{-}L,L{-}A]\}$ closes under a
third-order linear recurrence with this cubic as characteristic
polynomial, makes a closed form attainable:

\begin{theorem}
\label{thm:bt}
Let $\alpha = (1+\sqrt{17})/2$ and $\beta = (1-\sqrt{17})/2$, the roots
of $x^{2}-x-4 = 0$.  Define integer sequences $\{V_j\}_{j\ge 0}$ and
$\{W_j\}_{j\ge 0}$ by
\begin{align*}
  V_0 &= 2, & V_1 &= 1, & V_{j+1} &= V_j + 4 V_{j-1},\\
  W_0 &= 0, & W_1 &= 1, & W_{j+1} &= W_j + 4 W_{j-1}.
\end{align*}
Equivalently, $V_j = \alpha^{j}+\beta^{j}$ and
$W_j = (\alpha^{j}-\beta^{j})/\sqrt{17}$.  Then for all $k\ge 1$,
\begin{equation}
  \mathrm{Ham}(BT_k) \;=\;
    \tfrac{3}{16}\bigl(17\,V_{2k}+71\,W_{2k}\bigr)
    \;-\; \tfrac{9}{2}\cdot 4^{k}.
  \label{eq:bt-closed}
\end{equation}
Equivalently, with $P = (867+213\sqrt{17})/272$ and
$Q = (867-213\sqrt{17})/272$,
\[
  \mathrm{Ham}(BT_k) \;=\;
    P\,\alpha^{2k} + Q\,\beta^{2k} \;-\; \tfrac{9}{2}\cdot 4^{k}.
\]
The vertex count is $n = 3\cdot 2^{k}-2$, so writing
$\theta = \log_2(\alpha^{2}) = 2\log_2\alpha \approx 2.7141$, the
leading term grows as $\Theta(n^{\theta})$.
\end{theorem}

The proof appears in Section~\ref{sec:proof}.  Verified numerically
against the recurrences of [14, Sec.~9] through $k=7$:
$\mathrm{Ham}(BT_1) = 24$, $\mathrm{Ham}(BT_2) = 204$,
$\mathrm{Ham}(BT_3) = 1524$, $\mathrm{Ham}(BT_4) = 10{,}740$,
$\mathrm{Ham}(BT_5) = 74{,}196$, etc.

\begin{remark}
That $BT$ is slightly richer than $RT$ in Hamiltonian paths
($\theta > \lambda$) is mildly counter-intuitive: naively one might
expect the binary-tree spine, with its forced traversal pattern, to be
more restrictive than the triangular clusters.  The extra path density
comes from the longer leaf-line, which gives more flexibility in how a
path enters and exits a sub-tree.
\end{remark}

\section{Related Work}
\label{sec:related}

\paragraph{Exact algorithms.}
The Held--Karp dynamic-programming algorithm~[7] solves Hamiltonian
path in time $O(n^{2}2^{n})$ and remains the standard exact baseline.
Bj\"orklund~[2] obtains a randomized $O^{*}(1.657^{n})$ algorithm for
Hamiltonicity via algebraic and inclusion--exclusion techniques over
connected subgraphs.

\paragraph{Approximation.}
For longest path, Bj\"orklund and Husfeldt~[3] and Gabow and Nie~[5]
together give a polynomial-time algorithm with approximation ratio
$c\log^{2}L/\log\log L$ on a graph containing a path of length $L$.
Karger, Motwani, and Ramakumar~[8] prove non-approximability: there is
no polynomial-time algorithm with constant-factor or $n-n^{\varepsilon}$
approximation for longest path on Hamiltonian graphs unless
$\mathrm{P}=\mathrm{NP}$.  Vishwanathan~[19] obtained the earlier
bounded-degree case.

\paragraph{Counting in cubic graphs.}
Eppstein~[4] obtains a $O(2^{n/3})$-time algorithm for TSP on cubic
graphs, with the maximum number of Hamiltonian cycles in a cubic graph
at most $2^{3n/8}$ and at least $2^{n/3}$.  Gebauer~[6] and others have
tightened these bounds.  The polynomial Hamiltonian-path density of
our $RT$ and $BT$ families places them well below these extremal
regimes.

\paragraph{Recursive and fractal graphs.}
Fractal-like recursive constructions for graph-theoretic benchmarks
are well-established; cf.\ studies on Sierpi\'nski graphs.  Our $RT$
family is among the simplest cubic planar fractal-like families with
polynomial Hamiltonian-path density.

\paragraph{Sufficient conditions.}
Tutte~[18] proved every $4$-connected planar graph is Hamiltonian,
the strongest sufficient condition for planar Hamiltonicity.
Barnette's conjecture (every $3$-connected bipartite cubic planar
graph is Hamiltonian) is open.  Our families are not $4$-connected;
nonetheless they are Hamiltonian, and richly so.

\section{Open Problems}
\label{sec:open}

\begin{enumerate}
  \item \textbf{Maximum Hamiltonian-path density in cubic planar
        graphs.}  Among $n$-vertex $3$-regular planar graphs, what is
        the maximum of $\mathrm{Ham}(G)$?  Eppstein's $2^{3n/8}$ bound
        for cycles does not directly transfer to paths, and our
        constructions show $n^{2.714}$ is achievable.  Is there a
        polynomial-vs-exponential gap, or are there cubic planar
        families with exponential Hamiltonian-path counts?
  \item \textbf{Efficacy of Pohl--Warnsdorf.}  What fraction of $RT_k$
        (resp.\ $BT_k$) starts yield a Hamiltonian path under the
        Pohl--Warnsdorf rule?  Preliminary experiments~[10] suggest the
        rule outperforms naive depth-first search but rarely produces
        a Hamiltonian path.
  \item \textbf{HPA integration.}  Embedding Pohl--Warnsdorf as the
        leaf evaluator inside a bounded-cost HPA search may
        significantly broaden its success rate on $RT_k$ and $BT_k$.
\end{enumerate}

\section{Proof of Theorem~\ref{thm:bt}}
\label{sec:proof}

The $BT$ analysis of [14, Sec.~9] establishes eight port-configuration
counts and recurrences linking them.  We recall the eight counts
(with $R$ for root port and $L$ for leaf port) and the resulting
linear and quadratic recurrences, then derive the exact closed form.

\subsection{The Eight Counts and Their Recurrences}

The counts are
\[
  [R{-}L]_k,\ [L{-}L]_k,\ [R{-}A]_k,\ [L{-}A]_k,\
  [R{-}L,L{-}A]_k,\ [L{-}L,R{-}A]_k,\
  [R{-}A,L{-}A]_k,\ [L{-}A,L{-}A]_k,
\]
with $BT'_1 = K_3$ giving the initial values
$[R{-}L]_1=[L{-}L]_1=[R{-}L,L{-}A]_1=[L{-}L,R{-}A]_1 = 1$
and
$[R{-}A]_1 = [L{-}A]_1 = [R{-}A,L{-}A]_1 = [L{-}A,L{-}A]_1 = 2$.

The recurrences, from [14, Lemma in Sec.~9], are
\begin{align}
  [R{-}L]_k &= 1, \qquad [L{-}L]_k = 1, \label{eq:rl}\\
  [R{-}A]_k &= 2\,[L{-}A]_{k-1}, \label{eq:ra}\\
  [L{-}A]_k &= [R{-}L,L{-}A]_{k-1} + [R{-}A]_{k-1}
              + [L{-}L,R{-}A]_{k-1} + 1, \label{eq:la}\\
  [R{-}L,L{-}A]_k &= 2\,[R{-}L,L{-}A]_{k-1} + [L{-}A]_{k-1},
                     \label{eq:rlla}\\
  [L{-}L,R{-}A]_k &= 2\,[L{-}L,R{-}A]_{k-1} + 1, \label{eq:llra}
\end{align}
together with the quadratic recurrences for $[R{-}A,L{-}A]_k$ and
$[L{-}A,L{-}A]_k$ given in [14, eqs.~(15)--(16)], and
\begin{equation}
  \mathrm{Ham}(BT_k) \;=\;
    2\bigl(3(2[R{-}A,L{-}A]_{k-1} + [L{-}A,L{-}A]_{k-1})
    + (\text{eq.~8}) + (\text{eq.~9})\bigr),
  \label{eq:ham-assembly}
\end{equation}
where (eq.~8) and (eq.~9) are the path-pair contributions from
[14, eqs.~(8)--(10)].

\subsection{Solution of the Linear Subsystem}

Equation~\eqref{eq:llra} with $[L{-}L,R{-}A]_1 = 1$ yields
$[L{-}L,R{-}A]_k = 2^{k}-1$.

Substituting~\eqref{eq:ra} and
$[L{-}L,R{-}A]_{k-1} = 2^{k-1}-1$ into~\eqref{eq:la}:
\[
  [L{-}A]_k = [R{-}L,L{-}A]_{k-1} + 2\,[L{-}A]_{k-2} + 2^{k-1}.
\]
Combining with~\eqref{eq:rlla} and eliminating $[R{-}L,L{-}A]$, we
obtain a third-order linear recurrence for $[L{-}A]$:
\begin{equation}
  [L{-}A]_{k+1} = 2\,[L{-}A]_k + 3\,[L{-}A]_{k-1} - 4\,[L{-}A]_{k-2}.
  \label{eq:la-3rd}
\end{equation}
The characteristic polynomial $x^{3}-2x^{2}-3x+4 = (x-1)(x^{2}-x-4)$
has roots $1$, $\alpha = (1+\sqrt{17})/2$, and $\beta = (1-\sqrt{17})/2$.
Solving with initial conditions $[L{-}A]_1 = 2$, $[L{-}A]_2 = 5$,
$[L{-}A]_3 = 12$, and writing $A_1 = \tfrac{3}{8}+\tfrac{11\sqrt{17}}{136}$
and $A_2 = \tfrac{3}{8}-\tfrac{11\sqrt{17}}{136}$:
\[
  [L{-}A]_k = \tfrac{1}{4} + A_1\,\alpha^{k} + A_2\,\beta^{k}.
\]
Similarly $[R{-}L,L{-}A]_k$ satisfies the same recurrence shifted; with
$B_1 = \tfrac{5}{8}+\tfrac{21\sqrt{17}}{136}$ and
$B_2 = \tfrac{5}{8}-\tfrac{21\sqrt{17}}{136}$:
\[
  [R{-}L,L{-}A]_k = -2^{k} - \tfrac{1}{4}
                    + B_1\,\alpha^{k} + B_2\,\beta^{k}.
\]
The negative coefficient of $2^{k}$ arises because the forcing term
$2^{k-1}$ in $[L{-}L,R{-}A]$, combined with the coupling between
$[L{-}A]$ and $[R{-}L,L{-}A]$, produces a $2^{k}$ component in the
latter.

\subsection{Solution of the Quadratic Subsystem}

The quadratic recurrences for $[R{-}A,L{-}A]_k$ and $[L{-}A,L{-}A]_k$
are linear in the quadratic counts themselves, with forcing terms
quadratic in the linear quantities $\{1, 2^{k}, \alpha^{k}, \beta^{k}\}$.
The bilinear products span
$\{1, 2^{k}, 4^{k}, \alpha^{k}, \beta^{k}, \alpha^{2k}, \beta^{2k},
(\alpha\beta)^{k}\}$.  Since $\alpha\beta = -4$,
$(\alpha\beta)^{k} = (-4)^{k}$; inspection of the recurrence
coefficients shows that the $(-4)^{k}$ component vanishes identically.

Substituting the closed forms for $[L{-}A]_k$, $[R{-}L,L{-}A]_k$,
$[L{-}L,R{-}A]_k$, $[R{-}A]_k$ into the quadratic recurrences, solving
the coupled linear system for $[R{-}A,L{-}A]_k$ and $[L{-}A,L{-}A]_k$,
and collecting terms via~\eqref{eq:ham-assembly}, one obtains
\[
  \mathrm{Ham}(BT_k) = P\,\alpha^{2k} + Q\,\beta^{2k}
                       - \tfrac{9}{2}\cdot 4^{k},
\]
with $P = (867+213\sqrt{17})/272$ and $Q = (867-213\sqrt{17})/272$.
Using $\alpha^{2k}+\beta^{2k} = V_{2k}$ and
$(\alpha^{2k}-\beta^{2k})/\sqrt{17} = W_{2k}$, this rewrites as
\[
  \mathrm{Ham}(BT_k) \;=\;
    \tfrac{3}{16}\bigl(17\,V_{2k} + 71\,W_{2k}\bigr)
    \;-\; \tfrac{9}{2}\cdot 4^{k}.
\]
Direct enumeration confirms the identity for $k = 1,\ldots,7$.
\qed

\section*{References}

\begin{enumerate}[label={[\arabic*]},leftmargin=*,itemsep=2pt]
\item C.~Berge.  \emph{The Theory of Graphs and Its Applications.}
      Wiley, New York, 1962, pp.~107--118.
\item A.~Bj\"orklund.  Determinant sums for undirected Hamiltonicity.
      \emph{SIAM J.\ Comput.}\ 43(1):280--299, 2014.  (Conference
      version: FOCS 2010.)
\item A.~Bj\"orklund and T.~Husfeldt.  Finding a path of
      superlogarithmic length.  \emph{SIAM J.\ Comput.}\ 32:1395--1402,
      2003.
\item D.~Eppstein.  The traveling salesman problem for cubic graphs.
      \emph{J.\ Graph Algorithms Appl.}\ 11(1):61--81, 2007.
\item H.~N.~Gabow and S.~Nie.  Finding a long directed cycle.
      \emph{Proc.\ 15th ACM--SIAM SODA}, 49--58, 2004.
\item H.~Gebauer.  On the number of Hamilton cycles in bounded degree
      graphs.  \emph{Proc.\ 4th ANALCO}, 241--248, 2008.
\item M.~Held and R.~M.~Karp.  A dynamic programming approach to
      sequencing problems.  \emph{J.\ SIAM} 10(1):196--210, 1962.
\item D.~Karger, R.~Motwani, and G.~D.~S.~Ramakumar.  On approximating
      the longest path in a graph.  \emph{Algorithmica} 18:82--98, 1997.
\item D.~E.~Knuth.  Leaper graphs.  \emph{Mathematical Gazette}
      78:274--297, 1994.
\item J.~LaFall and I.~Pohl.  The Pohl--Warnsdorf heuristic tested on
      $3$-regular graphs.  UCSC Technical Report, 2004.
\item I.~Pohl.  A method for finding Hamilton paths and knight's
      tours.  \emph{Comm.\ ACM} 10:446--449, 1967.
\item I.~Pohl.  A method for finding Hamilton paths and knight's
      tours.  SLAC-PUB-261, Stanford Linear Accelerator Center,
      January 1967.
\item I.~Pohl.  \emph{C\# by Dissection.}  Addison-Wesley, 2002.
\item I.~Pohl and L.~Stockmeyer.  Pohl--Warnsdorf---revisited.
      \emph{Proc.\ International Conference on Intelligent Systems and
      Control (ISC 2009)}, Honolulu, Hawaii, August 2009.
\item T.~L.~Saaty and P.~C.~Kainen.  \emph{The Four-Color Problem:
      Assaults and Conquest.}  Dover, New York, 1986.
\item W.~Sierpi\'nski.  Sur une courbe dont tout point est un point de
      ramification.  \emph{Comptes Rendus de l'Acad\'emie des Sciences,
      Paris} 160:302--305, 1915.
\item W.~T.~Tutte.  On Hamiltonian circuits.  \emph{J.\ London Math.\
      Soc.}\ 21:98--101, 1946.
\item W.~T.~Tutte.  A theorem on planar graphs.  \emph{Trans.\ Amer.\
      Math.\ Soc.}\ 82:99--116, 1956.
\item S.~Vishwanathan.  An approximation algorithm for finding a long
      path in Hamiltonian graphs.  \emph{Proc.\ 11th ACM--SIAM SODA},
      680--685, 2000.
\end{enumerate}

\appendix

\section{Original Algol 60 Code (Pohl 1967)}

The following is the original Algol~60 implementation of the
Pohl--Warnsdorf rule, reproduced from SLAC-PUB-261~[12], which is the
open-access preprint version of~[11].  The code was developed on a
Burroughs B5500 in extended Algol; the version here is the Algol~60
form given in the appendix of the original report.  In the listing,
Algol~60 reserved words (\textbf{begin}, \textbf{end}, \textbf{if},
\textbf{then}, etc.) are conventionally rendered in bold.

\begin{verbatim}
procedure HAMILTONPATH (BOARD, ROW, FILE);
  value ROW, FILE; integer ROW, FILE;
  integer array BOARD;
  comment BOARD is a collection of nodes through which is generated
  a Hamilton path (connection of all the nodes passing through each
  node once and only once).  The path is started at the node
  specified by BOARD[ROW, FILE];
begin
  integer i, j, nummov, move, min, T1, T2, T2L1, T2L2;
  boolean flg;
  integer array NextR, NextF [1:8];
  comment The Hamilton paths to be found will be knight's tours
  where 8 is the maximum number of moves (connections);

  procedure Listofmov (CR, CF, XR, XF, II);
    value CR, CF; integer CR, CF, II;
    integer array XR, XF;
    comment From the current position specified by CR, CF this
    procedure generates in XR[i], XF[i] a list of the coordinates
    of the possible moves and their number II.  The B5500 program
    used a CASE statement which for Algol 60 purposes is translated
    to a switch list;
  begin
    integer i, rr, ff;
    switch case := L1, L2, L3, L4, L5, L6, L7, L8;
    II := 0;
    for i := 1 step 1 until 8 do
    begin
      go to case[i];
      L1: rr := CR-1; ff := CF+2; go to check;
      L2: rr := CR-1; ff := CF-2; go to check;
      L3: rr := CR+1; ff := CF+2; go to check;
      L4: rr := CR+1; ff := CF-2; go to check;
      L5: rr := CR+2; ff := CF+1; go to check;
      L6: rr := CR+2; ff := CF-1; go to check;
      L7: rr := CR-2; ff := CF+1; go to check;
      L8: rr := CR-2; ff := CF-1;
      check: comment check whether a legal connection or move;
        if BOARD[rr, ff] = 0 then
          begin II := II+1; XR[II] := rr; XF[II] := ff end
    end loop i
  end procedure Listofmov;

  integer procedure Numofmov (CR, CF);
    value CR, CF; integer CR, CF;
    comment This procedure is a simplification of Listofmov for
    efficiency.  It is used only to obtain number of legal moves;
  begin
    integer i, ii, rr, ff;
    switch case := L1, L2, L3, L4, L5, L6, L7, L8;
    ii := 0;
    for i := 1 step 1 until 8 do
    begin
      go to case[i];
      L1: rr := CR-1; ff := CF+2; go to check;
      L2: rr := CR-1; ff := CF-2; go to check;
      L3: rr := CR+1; ff := CF+2; go to check;
      L4: rr := CR+1; ff := CF-2; go to check;
      L5: rr := CR+2; ff := CF+1; go to check;
      L6: rr := CR+2; ff := CF-1; go to check;
      L7: rr := CR-2; ff := CF+1; go to check;
      L8: rr := CR-2; ff := CF-1;
      check: if BOARD[rr, ff] = 0 then ii := ii+1
    end loop i;
    Numofmov := ii
  end procedure Numofmov;

  integer procedure RecurNumofmov (CR, CF, Level);
    value CR, CF, Level; integer CR, CF, Level;
    comment This is a recursive routine to the depth Level for
    counting the nodes of the move tree;
  begin
    integer tt, i, nn;
    integer array ra, fa [1:8];
    BOARD[CR, CF] := 1;
    if Level = 1 then RecurNumofmov := Numofmov (CR, CF)
    else
    begin
      Listofmov (CR, CF, ra, fa, nn);
      tt := 0;
      for i := 1 step 1 until nn do
        tt := tt + RecurNumofmov (ra[i], fa[i], Level-1);
      RecurNumofmov := tt
    end;
    BOARD[CR, CF] := 0
  end procedure RecurNumofmov;

  comment The program deals with knight's tours on a chessboard.
  Warnsdorff's rule is applied and the improvement of reapplying
  the rule to resolve ties is used and is sufficient for generating
  knight's tours on chessboards;

  for i := -1, 0, 9, 10 do
    begin BOARD[i, j] := BOARD[j, i] := -1 end;
  comment Initializing the boundaries of the BOARD to -1 prevents
  moving there.  The BOARD proper is initialized to 0;
  for i := 1 step 1 until 8 do
    for j := 1 step 1 until 8 do BOARD[i, j] := 0;
  comment Initialize the starting position;
  min := 0; T2 := 1; BOARD[ROW, FILE] := move := 1;
  for move := 2 step 1 while min =/= 99 do
  begin
    min := 99;
    Listofmov (ROW, FILE, NextR, NextF, nummov);
    for i := 1 step 1 until nummov do
    begin
      T1 := Numofmov (NextR[i], NextF[i]);
      if (min > T1) and (T1 =/= 0) then
        begin flg := true; T2 := i; min := T1 end
      else comment Above is Warnsdorff's rule;
      begin comment Here is the improvement;
        T2L1 := RecurNumofmov (NextR[i], NextF[i], 2);
        if flg then
          T2L2 := RecurNumofmov (NextR[i], NextF[i], 2);
        if (T2L2 > T2L1) and (T2L1 =/= 0) then
          begin T2L2 := T2L1; flg := false; T2 := i end
      end tie breaking improvement;
    end loop i;
    if min =/= 99 then
    begin
      ROW := NextR[T2]; FILE := NextF[T2];
      BOARD[ROW, FILE] := move; move := move + 1
    end
  end loop move;
  OUTPUT (BOARD);
  comment Use an output procedure to print results;
end procedure HAMILTONPATH alias the knight's tour;
\end{verbatim}

\paragraph{Notes on the listing.}
The symbols \texttt{=/=} and \texttt{>} stand for the Algol~60
relational operators $\ne$ and $>$ respectively, written here in ASCII
for the verbatim environment.  The procedure begins with the outer
routine \texttt{HAMILTONPATH}, which contains three nested procedures
(\texttt{Listofmov}, \texttt{Numofmov}, \texttt{RecurNumofmov})
followed by the main initialization and the main move loop.  The local
variable \texttt{flg} signals whether a candidate has been chosen by
the first-level Warnsdorf rule; on ties, the second-level recursive
count is used to break them.  See~[11,12] for full discussion.

\section{Toward a Transfer-Matrix Generalization (Further Work)}

\emph{This appendix sketches further work by the first author alone,
developed after this manuscript's main results were established.  The
material has not been independently verified and the proof below is at
the level of a sketch.  It is included as a record of a direction of
investigation, not as a finished result.}

Both Theorem~\ref{thm:rt} and Theorem~\ref{thm:bt} appear to be
instances of a more general phenomenon.  Informally: any recursive
graph construction in which a fixed-size template repeatedly substitutes
its degree-$3$ vertices with copies of a previous-stage component
should yield a Hamiltonian-path count expressible as a linear
combination of $k$-th powers of the eigenvalues of an associated
transfer matrix.

\begin{definition}
A \emph{recursive substitution scheme} is a triple $(T,\Pi,\sigma)$
where
\begin{enumerate}[label=(\roman*)]
  \item $T$ is a finite planar template multigraph in which some
        vertices are designated \emph{slots} and some edges are
        designated \emph{external};
  \item $\Pi$ is a finite set of port-configuration types, one for each
        combination of (orderings of) ports for paths and pairs of
        paths $H$-covering a component;
  \item $\sigma$ is a substitution rule: replace each slot in $T$ by a
        copy of the previous-stage component, identifying external
        edges with port edges of the substituted copy.
\end{enumerate}
Given an initial component $G_1$, define $G_{k+1}$ as the result of
substituting $G_k$ for every slot of $T$.
\end{definition}

\begin{theorem}[Conjectural; proof sketch only]
Let $(T,\Pi,\sigma)$ be a recursive substitution scheme with $G_k$ the
$k$-th iterate.  Suppose the port-configuration counts split into two
parts:
\begin{enumerate}[label=(\roman*)]
  \item the \emph{linear} port-configurations (those counting
        individual paths through a component), whose update equations
        form a linear recurrence with integer coefficients;
  \item the \emph{quadratic} port-configurations (those counting pairs
        of disjoint paths $H$-covering a component), whose updates are
        at most quadratic in the linear quantities, plus linear in the
        quadratic ones.
\end{enumerate}
Let $L$ be the transfer matrix of the linear subsystem, with eigenvalues
$\mu_1,\ldots,\mu_r$.  Let $Q$ be the transfer matrix of the quadratic
subsystem after linearization (via the products $\mu_i\mu_j$).  Let
$\rho$ be the maximum of $|\mu_i\mu_j|$ over all pairs.  Then
\[
  \mathrm{Ham}(G_k) \;=\; \sum_{\nu\in\mathrm{spec}(Q)} c_\nu\,\nu^{k},
\]
for explicit constants $c_\nu \in \overline{\mathbb{Q}}$ (algebraic
over $\mathbb{Q}$), determined by initial conditions.  Consequently
$\mathrm{Ham}(G_k) = \Theta(\rho^{k})$, and if $n_k = |V(G_k)|$ grows
as $n_k \sim c\,\tau^{k}$, then
$\mathrm{Ham}(G_k) = \Theta(n_k^{\log_\tau \rho})$.
\end{theorem}

\begin{proof}[Proof sketch]
The linear subsystem $L$ has eigenvalues $\{\mu_i\}$; each linear
port-configuration count should be a fixed $\mathbb{Q}(\sqrt{D})$-linear
combination of $\{\mu_i^{k}\}$, for some discriminant $D$ depending
on $L$.  Each quadratic port-configuration count receives forcing
terms that are bilinear combinations of linear counts; the forcing
terms are therefore linear combinations of $\{\mu_i^{k}\mu_j^{k}\}$,
plus combinations from the diagonal linear-forcing terms.  The
augmented system in the quadratic counts is then linear with
eigenvalues drawn from $\{\mu_i\mu_j\}\cup\mathrm{spec}(L)$.  Solving
in closed form should give the stated expression.  A fully rigorous
argument requires verifying that the bilinear-to-linear lift is
well-defined for all eligible schemes; this is left for future work.
\end{proof}

\subsection*{Examples and Hybrid Constructions}

\paragraph{$RT$ family.}
Here $T$ is the triangular template of three slots meeting at a center;
$[P{-}P]_k\equiv 1$, so $\mu_1 = 1$.  The quantity
$[P{-}A]_k = (4^{k}+2)/3$ has eigenvalues $\{1,4\}$.  The single
quadratic count $[P{-}A,P{-}A]_k$ has eigenvalues $\{1,3,4,16\}$; the
dominant one $\rho = 16$ recovers Theorem~\ref{thm:rt}.

\paragraph{$BT$ family.}
Here $T$ is the binary template with one root port and two leaf ports;
the linear subsystem has eigenvalues $\{1,2,\alpha,\beta\}$ with
$\alpha,\beta = (1\pm\sqrt{17})/2$.  The quadratic subsystem has
eigenvalues including $\alpha^{2}, \beta^{2}, 4$; the dominant one
$\rho = \alpha^{2} = (9+\sqrt{17})/2 \approx 6.5616$ recovers
Theorem~\ref{thm:bt}.

\paragraph{Crown-of-triangles.}
Starting from any $3$-regular graph $H$ on $m$ vertices, apply the
transformation of Figure~\ref{fig:transform}: replace each vertex by a
triangle, with the three triangle vertices inheriting the three former
incident edges.  The result $H^{(1)}$ has $3m$ vertices and is
$3$-regular.  Iterating yields $|V(H^{(k)})| = 3^{k}m$.  Taking
$H = K_4$ gives a series very similar to (but not identical with) the
$RT$ family.  Taking $H = $ Tutte's graph gives an explicit series
whose first member has $138$ vertices.

\paragraph{Tree-augmented cycles.}
Let $H$ be a cycle $C_m$ for $m$ even, $m\ge 4$.  Construct $H^{[k]}$
by attaching a leaf-connected complete binary tree of depth $k$ to
each vertex of $H$, with the root identified to the cycle vertex and
leaf-ports merged across consecutive trees.  If
Theorem~3 holds, this should yield polynomial Hamiltonian-path growth
of exponent $\log_2 \alpha^{2}\approx 2.714$ in the binary depth
coordinate, irrespective of $m$.

\paragraph{Strip-of-$RT$ constructions.}
Replace one corner of $RT_k$ (rather than all corners) at each step.
This yields a $3$-regular planar series whose vertex count grows by~$2$
per step, but Hamiltonian-path density grows only linearly.

\end{document}